\documentclass[10pt]{amsart}

\usepackage[colorlinks]{hyperref}
\usepackage{color,graphicx,shortvrb}
\usepackage[latin 1]{inputenc}

\usepackage[active]{srcltx} 

\usepackage{enumerate}

\newtheorem{theorem}{Theorem}[section]

\newtheorem{corollary}[theorem]{Corollary}

\newtheorem{lemma}[theorem]{Lemma}

\newtheorem{proposition}[theorem]{Proposition}
\newtheorem{remark}[theorem]{Remark}

\def\J#1#2#3{ \left\{ #1,#2,#3 \right\} }
\def\RR{{\mathbb{R}}}

\def\11{\textbf{$1$}}

\DeclareGraphicsExtensions{.jpg,.pdf,.png,.eps}

\usepackage{enumerate}

\begin{document}

\title{Positive definite hermitian mappings associated to tripotent elements}
\date{}

\author[A.M. Peralta]{Antonio M. Peralta}
\email{aperalta@ugr.es}
\address{Departamento de An{\'a}lisis Matem{\'a}tico, Facultad de
Ciencias, Universidad de Granada, 18071 Granada, Spain.}

\thanks{Authors partially supported by the Spanish Ministry of Economy and Competitiveness,
D.G.I. project no. MTM2011-23843, and Junta de Andaluc\'{\i}a grants FQM0199 and
FQM3737.}



\date{February, 2013}

\maketitle
 \thispagestyle{empty}

\begin{abstract} We give a simple proof of a meaningful result established by Y. Friedman and B. Russo in 1985, whose proof was originally based on strong holomorphic results. We provide a simple proof which is directly deduced from the axioms of JB$^*$-triples with techniques of Functional Analysis.
\end{abstract}

\section{Introduction}

In holomorphic theory, the Riemann mapping theorem states that every simply connected open proper subset of the complex plane is biholomorphically equivalent to the open unit disk. The theory of holomorphic functions of several complex variables is substantially different by many reasons, for example, as noted by Poincaré in the early 1900s, the Riemann mapping theorem fails when the complex plane is replaced by a complex Banach space of higher dimension. Though, a complete holomorphic classification of bounded simply connected domains in arbitrary complex Banach spaces is unattainable, {bounded symmetric domains} in finite dimensions were studied and classified by E. Cartan \cite{Cartan35} using the classification of simple complex Lie algebras, and by M. Koecher \cite{Ko} and O. Loos \cite{Loos} with more recent techniques of Jordan algebras and Jordan triple systems. A domain $\mathcal{D}$ in a complex Banach space $X$ is \emph{symmetric} if for each $a$ in  $\mathcal{D}$ there is a biholomorphic mapping $\Phi_a : \mathcal{D} \to  \mathcal{D}$; with $\Phi_a = \Phi_a^{-1}$, such that $a$ is an isolated fixed point of $\Phi_a$ (cf. \cite{Up} and \cite{Chu2012}). In a groundbreaking contribution, W. Kaup shows, in \cite{Ka83}, the existence of a set of algebraic-geometric-analytic axioms which determine a class of complex Banach spaces, the class of JB$^*$-triples, whose open unit balls are bounded symmetric domains, and every bounded symmetric domain in a complex Banach space is biholomorphically equivalent to the open unit ball of a
JB$^*$-triple; in this way, the category of all bounded symmetric domains with base point is equivalent to the category of JB$^*$-triples (see definitions below).

The dual ``holomorphic''-``geometric-analytic'' nature of JB$^*$-triples allowed different strategies to prove the most significant results in this category of complex Banach spaces. For example, the contractive projection principle asserting that the class of JB$^*$-triples is stable under contractive projections, was independently proved with holomorphic techniques by W. Kaup \cite{Ka84} and L.L. Stacho \cite{Sta82} and with tools of Functional Analysis by Y. Friedman and B. Russo \cite{FriRu87}. Though most of the most significative results in JB$^*$-triple theory have been established with dual holomorphic and analytic techniques, some important structure results remain unproved with techniques of Functional Analysis. An example of the latter is an useful property stated by Y. Friedman and B. Russo in their study of the structure of the predual of a JBW$^*$-triple carried out in \cite[Lemma 1.5]{FriRu}. Before going into details, we recall some background. It follows from the algebraic axioms in the definition of JB$^*$-triples,  each tripotent $e$ (i.e. $e= \J eee$) in a JB$^*$-triple, $E$, induces a \emph{Peirce decomposition} of $E$, $$E= E_{2} (e) \oplus E_{1} (e) \oplus E_0 (e),$$ where for $i=0,1,2$ $E_i (e)$ is the $\frac{i}{2}$ eigenspace of the mapping $L(e,e)(x) = \J eex$. Triple products between elements in Peirce subspaces satisfy the following Peirce multiplication rules:
$\J {E_{i}(e)}{E_{j} (e)}{E_{k} (e)}$ is contained in $E_{i-j+k} (e)$ if $i-j+k \in \{ 0,1,2\}$ and is zero otherwise. In addition, $$\J {E_{2} (e)}{E_{0}(e)}{E} = \J {E_{0} (e)}{E_{2}(e)}{E} =0.$$ It is further known from the axioms that $E_2 (e)$ is a JB$^*$-algebra with product $a\circ_{e} b := \J aeb$ and involution $a^{*_e} := \J eae$ (cf. \cite[\S 3]{Loos}, \cite[\S 19, \S 21]{Up} or \cite[\S 1.2 and Remark 3.2.2]{Chu2012}). Therefore, the mapping $$F_1 : E_1 (e) \times E_1 (e) \to E_2(e)$$ $$(x,y) \mapsto F_1 (x,y) =\J xye,$$ is well defined, continuous and sesquilinear. In \cite[Lemma 1.5]{FriRu}, Friedman and Russo state that $F_1$ also satisfies the following properties: \begin{enumerate}[(a)]\item $F_1$ is hermitian, i.e, $F_1(x,y)^{*_e} = Q(e) F_1 (x,y) = F_1 (y,x)$, for every $x,y\in E_1(e)$;
\item $F_1$ is positive definite: $\Phi (x,x) \geq 0$ in $E_2(e)$ for every $x\in E_1 (e)$, and $F_1 (x,x) =0$ implies $x=0$ in $E_1(e)$. $\hfill\Box$
\end{enumerate} The proof of the last statement in (b) is attributed to H. Upmeier in \cite{Up} (and to O. Loos \cite[10.4]{Loos} in the finite dimensional case). When exploring the last reference, we are led to a deep holomorphic argument which requires a firm background on holomorphic functions and Symmetric Banach Manifolds. Motivated by a question from my colleagues M. Cabrera and A. Rodríguez-Palacios, while they were gathering information for a book in preparation \cite{CaPal}, this note provides a simple proof of the above properties which is directly deduced from the axioms of JB$^*$-triples, and is free of holomorphic theory. We complement the original work of Friedman and Russo providing a proof derived from the axioms with techniques of Functional Analysis.

\section{The results}

A \emph{JB$^*$-triple} is a complex Banach space $E$ equipped with a
triple product $\{\cdot,\cdot,\cdot\}:E\times E\times E\rightarrow
E$ which is linear and symmetric in the outer variables, conjugate
linear in the middle one and satisfies the following conditions:
\begin{enumerate}[(JB$^*$-1)]
\item (Jordan identity) for $a,b,x,y,z$ in $E$,
$$\{a,b,\{x,y,z\}\}=\{\{a,b,x\},y,z\}
-\{x,\{b,a,y\},z\}+\{x,y,\{a,b,z\}\};$$
\item $L(a,a):E\rightarrow E$ is an hermitian
(linear) operator with non-negative spectrum, where $L(a,b)(x)=\{a,b,x\}$ with $a,b,x\in
E$;
\item $\|\{x,x,x\}\|=\|x\|^3$ for all $x\in E$.
\end{enumerate}

Complex vector spaces admitting a triple product satisfying the above Jordan identity are called Jordan triple systems. Examples of JB$^*$-triples include C$^*$-algebras, JB$^*$-algebras and the space of all bounded linear operators between complex Hilbert spaces.

Given an element $a$ in a Jordan triple system $E$, the symbol
$Q(a)$ will denote the conjugate linear mapping on $E$ defined by $Q(a) (b)
:= \J aba.$ It is known that the identity
$$ Q(a) Q(b) Q(a) = Q(Q(a)b),$$ holds for every $a,b\in E$.

It follows from Peirce arithmetic that for each tripotent $e$ in a JB$^*$-triple $E,$ the mapping $$F_1 : E_1 (e) \times E_1 (e) \to E_2(e)$$ $$F_1 (x,y) =\J xye,$$ is well defined, continuous and sesquilinear.

Let us recall some basic facts about numerical range and hermitian operators.
Let $X$ be a Banach space, and $u$ a norm-one element in $X$. The
set of states of $X$ relative to $u$, $D(X,u)$, is defined as the
non empty, convex, and weak*-compact subset of $X^{*}$ given by
$\displaystyle D(X,u) := \{ \phi \in B_{X^{*}} : \phi (u)=1 \}.$ For $x\in X$,
the \emph{numerical range} of $x$ relative to $u$, $V(X,u,x)$, is
defined by $V(X,u,x):= \{ \phi (x) : \phi \in D(X,u) \}$.
It is well known that a bounded linear operator $T$ on a
complex Banach space $X$ is hermitian if and only if $V( BL (X),
I_{X}, T) \subseteq \RR$, where $BL(X)$ denotes the Banach space of
all bounded linear operators on $X$ (compare \cite[Corollary 10.13]{BoDu}).
Furthermore, for each $T$ in $BL(X)$ we have
$V(BL(X),I_{X},T) = \overline{co} \ W(T)$ where $W(T) = \{ x^{*}
(T(x)) : (x,x^{*}) \in \Gamma \},$ and $\Gamma$ is any subset of $\Pi(X):=\{
(x,x^{*}) : x\in {X}, x^{*} \in {X^{*}}, x^{*} (x) =1= \|x\| =\|x^*\| \}$
satisfying that its projection onto the first coordinate is norm
dense in the unit sphere of ${X}$ (see \cite[Theorem 9.3]{BoDu1}).

We are particularly interested in the set $D(E_{2}(e),e)$ of all states of the JB$^*$-algebra $E_2(e)$. It is known that $D(E_{2}(e),e)$ separates the points of $E_2 (e)$, while an element $x\in E_2(e)$ is symmetric (i.e. $x^{*_e} = \J exe = x$) if and only if $\varphi (x) \in \mathbb{R}$ for every $\varphi \in D(E_{2}(e),e)$ (compare \cite[\S 1.2 and \S 3.6]{HanStor}). Fix $x\in E_1(e)$. The axiom (JB$^*$-1) implies that $L(x,x)$ is hermitian with non-negative spectrum, thus $\varphi L(x,x) (e)$ lies in $\mathbb{R}_0^+,$ for every $\varphi\in D(E_{2}(e),e).$ In particular $\J xxe$ is a positive element in $E_2(e)$, for every $x\in E_1(e)$. The polarisation formula $\J xye + \J yxe = \J {x+ y}{x+ y}e -\J xxe -\J yye$ ($x,y\in E_1(e)$), together with the above fact, shows that $\J xye + \J yxe$ is a symmetric element in the JB$^*$-algebra $E_2 (e)$. In other words, the mapping $F_1$ is hermitian and semi-definite positive. A similar argument was applied in \cite[Proposition 1.2]{BarFri} and in \cite[page 412]{PeRo2000}, to show that for every norm-one functional $\varphi\in E^*$ and every norm-one element $z\in E$ (respectively, $\Phi\in D( BL(E),I_{E})$), the mapping $(.\vert .)_{\varphi} : E\times E \to \mathbb{C}$ $(x\vert y)_{\varphi} := \varphi L(x,y)(z)$ (respectively, $(x\vert y)_{\Phi}:=\Phi L(x,y)$) defines a continuous semi-positive sesquilinear form on $E.$

Following standard notation, for each element $a$ in a JB$^*$-triple $E$ we denote $a^{[1]} =
a$ and $a^{[2 n +1]} := \J a{a^{[2n-1]}}a$ $(\forall n\in \mathbb{N})$. It is known that Jordan triples
are power associative, that is, $\J{a^{[k]}}{a^{[l]}}{a^{[m]}}=a^{[k+l+m]}$ (cf. \cite[\S 3.3]{Loos} or \cite[Lemma 1.2.10]{Chu2012} or simply apply the Jordan identity). The element $a$ is called \emph{nilpotent} if $a^{[2n+1]}=0$ for some $n$. A Jordan triple $E$ for which the vanishing of $\J aaa$ implies that $a$ itself vanishes is said to be \emph{anisotropic}. It is easy to check that $E$ is anisotropic if and only if zero is the unique nilpotent element in $E$. Clearly, every JB$^*$-triple is anisotropic.

\begin{lemma}\label{element in E1 and nilpotent} Let $a$ be an element in the Peirce-1 subspace associated to a tripotent, $e,$ in a JB$^*$-triple $E$. Suppose that $\J aae =0$. Then $\J bce =0$ and $ Q(b) Q(c) (e)=0$, for every $b,c$ in the JB$^*$-subtriple of $E$ generated by $a$. In particular\hyphenation{particular}, $\J {a^{[3]}}e{a^{[3]}} =0$.
\end{lemma}

\begin{proof} Fix an arbitrary state $\psi$ of the JB$^*$-algebra $E_2(e)$. The form $(\cdot \vert \cdot)_{\psi} : E\times E \to \mathbb{C},$ $(x\vert y )_{\psi}:= \psi \J xye$ is sesquilinear and semi-positive with $(a\vert a )_{\psi}=0$. By the Cauchy-Schwarz inequality, $(a\vert x )_{\psi}=(x\vert a )_{\psi}=0,$ for every $x\in E$. Since, $\psi$ was arbitrarily chosen in the set of states of $E_2 (e)$ and $D(E_{2}(e),e)$ separates the points of $E_2 (e)$, we deduce that \begin{equation}\label{eq l21} \J bae =  \J abe= 0,
\end{equation} for every $b$ in $E_1 (e)$ and, in particular, for every $b$ in the JB$^*$-subtriple, $E_a$, generated by $a$ (note that $E_a \subseteq E_1 (e)$).

Fix $b\in E_a$. We consider now products of the form $$\J {a^{[3]}}be = \J aa{\J abe} + \J a{\J aab}e - \J ab{\J aae}$$ $$= \hbox{ (by $(\ref{eq l21})$ and the hypothesis) } = 0.$$ Since $$\J {a^{[2 n+1]}}be = \J aa{\J {a^{[2n-1]}}be} + \J {a^{[2n-1]}}{\J aab}e - \J {a^{[2n-1]}}b{\J aae},$$  it follows by induction that $\J {a^{[2 n-1]}}be = 0$, for every natural $n$. By linearity and continuity we have \begin{equation}\label{eq l22} \J cbe =0,
\end{equation} for every $b,c\in E_a$, which proves the first statement.

For the second statement we consider $b,c\in E_a$. By $(\ref{eq l22})$, $$Q(b) Q(c) (e) = \J b{\J cec}b = - \J ec{\J bcb} + 2 \J {\J ecb}cb = 0.$$ Consequently, $\J {a^{[3]}}e{a^{[3]}} = Q(a^{[3]}) (e) =Q(a) Q(a)^2 (e) =0$, as we wanted.
\end{proof}

\begin{remark}\label{remark 1}{\rm Let $e$ be a tripotent in a JB$^*$-triple $E$ and let $a$ be an element in $E_1 (e)$ with $\J aae =0$. The proof of the above Lemma \ref{element in E1 and nilpotent} actually shows that $\J bce =0$, whenever $b,c $ belong to $E_1 (e)$ and one of them lies in the JB$^*$-subtriple of $E$ generated by $a$.}
\end{remark}

We can give now an elementary proof of the fact that the mapping $F_1$ is positive definite on $E_1 (e)\times E_1 (e)$.

\begin{lemma}\label{l positivity} Let $e$ be a tripotent in a JB$^*$-triple $E$. Then the sesquilinear mapping $F_1 : E_1 (e) \times E_1 (e) \to E_2(e),$ $F_1(a,b)= \J abe,$ is hermitian and positive definite, that is, $F_1 (a,b)^{*_e} = F_1(b,a)$, $F_1 (a,a)\geq 0$ in $E_2 (e),$ and
$F_1 (a,a)=0$ if and only if $a=0$.
\end{lemma}

\begin{proof} We only have to prove the last statement. To this end, take an element $a$ in $E_1(e)$ with $\J aae =0$. The Jordan identity implies that $$\frac12 a^{[7]} =\frac12 \J {a^{[3]}}a{a^{[3]}} = \J {a^{[3]}}{\J eea}{a^{[3]}}$$ $$ = -  \J ea{\J {a^{[3]}}{e}{a^{[3]}}} + 2 {\J {\J ea{a^{[3]}}}{e}{a^{[3]}}}.$$
By Lemma \ref{element in E1 and nilpotent} and the assumptions, the last two summands in the above identity are zero, therefore $ a^{[7]} =0,$ which gives $a=0$.
\end{proof}

\begin{remark}\label{r Lusin}{\rm An strengthened geometric version of the above Lemma \ref{l positivity} (\cite[Lemma 1.5]{FriRu} was proved in \cite[Proposition 2.4]{BuFerMarPe}, where it is established that for each tripotent $e$ in a JB$^*$-triple $E$ and every $x\in E_1(e)\cup E_2 (e)$ we have $\|x\|^2 \leq 4 \|\J xxe\|.$ However, the proof of this result makes use of the original result by Friedman and Russo and subsequent strong structure results proved by the same authors in \cite{FriRu86}.}
\end{remark}

We shall finish with an observation. The above Lemma \ref{l positivity} was established for finite dimensional JB$^*$-triples by O. Loos in \cite[10.4]{Loos}. In such case, the proof is based on the following fact: for every finite dimensional JB$^*$-triple $V$ there exists an $Aut(V)$-invariant positive definite hermitian scalar product $(.\vert .) : V \times V \to \mathbb{C}$ such that $(L(x,y) (a) \vert b) = (a \vert L(b,a) (c))$ for every $a,b,x,y \in V$ (cf. \cite[Proposition 3.4 and Lemma 10.2]{Loos}, see also \cite[\S 1.2]{Chu2012}). 
Although, for an infinite dimensional JB$^*$-triple $E$, the existence of positive definite hermitian scalar product on $E$ satisfying the above condition seems to be hopeless, we shall give an argument to guarantee its existence at least locally.

For a general JB$^*$-triple $E,$ T. Barton and Y. Friedman showed, in \cite[Proposition 1.2]{BarFri}, the abundance of semi-positive sesquilinear forms on $E$. This fact will follow as a consequence of our last result.

Let $c$ be an element in a JB$^*$-triple $E$. Applying the Commutative Gelfand Theory for JB$^*$-triples (see \cite[\S 1]{Ka83}), the JB$^*$-subtriple, $E_c$, of $E$ generated by $c$ is JB$^*$-isomorphic to $C_0(S)$ for some locally compact Hausdorff space $S\subseteq (0,\|c\|]$, such that $S\cup \{0\}$ is compact. It is also known that there exists a triple isomorphism $\Psi$ from $E_c$ onto $C_{0}(S),$ such that $\Psi (c) (t) = t$ $(t\in S)$ (compare \cite[Lemma 1.14]{Ka83}). 

\begin{proposition}\label{p Royden} Let $a$ be a norm-one element in a JB$^*$-triple $E$ and let $\varphi$ be a norm-one functional satisfying $\varphi (a) = 1$. Then the restriction of $\varphi$ to the JB$^*$-subtriple of $E$ generated by $a$ is a triple homomorphism.
\end{proposition}

\begin{proof} Let $E_a$ denote the JB$^*$-subtriple of $E$ generated by $a$. We can identify $E_a$ with $C_{0}(S),$ where $S$ is a locally compact Hausdorff space contained in $(0,1]$, $1\in S$ and $S\cup \{0\}$ is compact. Further, under this identification we can assume that $a$ identifies with the mapping $t\mapsto t$ ($t\in S$). Let $\psi = \varphi|_{E_a}.$ Clearly, $\|\psi\| = 1 = \psi (a).$ By a result of H.L. Royden (see \cite{Royden}), there exists a non-negative (regular) Baire measure $\mu$ on $S\cup \{0\}$ with $\mu (S\cup \{0\}) =1$, such that $|a|\equiv 1$ on the closed support of $\mu$ and $\psi (g) = \int_{S\cup \{0\}} g(t) \overline{a(t)} d\mu(t)$ ($g\in C_{0}(S)$). Clearly, $\psi$ coincides with the functional $\delta_{1} (g) = g(1)$  ($g\in C_{0}(S)$), and the statement follows.
\end{proof}

\begin{corollary}\label{c Royden} Let $a$ be a norm-one element in a JB$^*$-triple $E$. For each norm-one functional satisfying $\varphi$ in $E^*$ with $\varphi (a) = 1$, the prehilbertian seminorm $(x\vert y)_{\varphi} =\J xya$ {\rm(}$x,y\in E${\rm )} satisfies $(a\vert a)_{\varphi} =1.$ $\hfill\Box$
\end{corollary}

\bigskip
\bigskip
\bigskip
\bigskip
\bigskip
\bigskip

\medskip

\noindent Departamento de An\'{a}lisis Matem\'{a}tico, Facultad de
Ciencias,\\ Universidad de Granada, 18071 Granada, Spain. \medskip

\noindent e-mail: aperalta@ugr.es

\end{document}